\documentclass[11pt]{article}

\usepackage{amsmath}
\usepackage{amssymb}
\usepackage{amscd}
\usepackage{amsthm}
\usepackage{indentfirst}

\newcommand{\Z}{\ensuremath{\mathbb{Z}}}
\newcommand{\Q}{\ensuremath{\mathbb{Q}}}

\usepackage[margin=2.6cm]{geometry}

\theoremstyle{plain}
\newtheorem{thm}{Theorem}[section]
\newtheorem{lem}[thm]{Lemma}
\newtheorem{cor}[thm]{Corollary}

\newtheorem{question}[thm]{Question}

\newtheorem{facts}[thm]{Facts}
\theoremstyle{definition}
\newtheorem{dfn}[thm]{Definition}

\newtheorem{rmq}[thm]{Remark}

\DeclareMathOperator{\Spec}{Spec}
\DeclareMathOperator{\Pic}{Pic}
\DeclareMathOperator{\divisor}{div}
\DeclareMathOperator{\disc}{disc}
\DeclareMathOperator{\Res}{Res}

\DeclareMathOperator{\Ext}{Ext}
\DeclareMathOperator{\fd}{fd}

\DeclareMathOperator{\cont}{cont}

\DeclareMathOperator{\PQuad}{PrimQuad}

\newcommand{\PP}{\mathbb{P}}
\newcommand{\A}{\mathbb{A}}
\newcommand{\W}{\mathcal{W}}
\newcommand{\V}{\mathcal{V}}
\newcommand{\X}{\mathcal{X}}
\newcommand{\J}{\mathcal{J}}
\newcommand{\Poinc}{\mathcal{P}}
\newcommand{\Gm}{\mathbf{G}_{{\rm m}}}
\newcommand{\GL}{\mathrm{GL}}

\newcommand{\cl}{\mathrm{cl}}

\begin{document}

\title{From Picard groups of hyperelliptic curves to class groups of quadratic fields}

\author{Jean Gillibert}

\date{December 2020}

\maketitle

\begin{abstract}
Let $C$ be a hyperelliptic curve defined over $\Q$, whose Weierstrass points are defined over extensions of $\Q$ of degree at most three, and at least one of them is rational. Generalizing a result of R. Soleng (in the case of elliptic curves), we prove that any line bundle of degree $0$ on $C$ which is not torsion can be specialised into ideal classes of imaginary quadratic fields whose order can be made arbitrarily large. This gives a positive answer, for such curves, to a question by Agboola and Pappas.
\end{abstract}


\section{Introduction}


Let $C$ be a smooth projective geometrically connected curve defined over $\Q$. The motivation for this paper is to address the following question, which is a special case of a question raised by Agboola and Pappas in \cite{ap00}.

\begin{question}
\label{mainQ}
Given a non-trivial line bundle $L\in\Pic^0(C)$, is it possible to find points $P\in C(\overline{\Q})$ for which the specialisation of $L$ at $P$ is a non-trivial ideal class of the number field $\Q(P)$?
\end{question}

Some notation: given a point $P\in C(\overline{\Q})$, we denote by $\Q(P)$ the field of definition of $P$; given a number field $K$, we denote by $\mathcal{O}_{K}$ its ring of integers.

The question above is deliberately formulated in a vague sense. Indeed, given a line bundle $L\in\Pic^0(C)$ and a point $P\in C(\overline{\Q})$, there is no canonical way to specialise $L$ into a line bundle on $\Spec(\mathcal{O}_{\Q(P)})$, \emph{i.e.} into an ideal class of $\Q(P)$.

As one can guess, this technical issue is linked to bad reduction phenomena for the curve $C$. There are essentially two approaches which allow to overcome this:

\begin{enumerate}
\item[(A1)] invert primes of bad reduction of $C$;
\item[(A2)] restrict to line bundles $L$ which correspond to points with everywhere good reduction in the Jacobian variety $J$ of $C$.
\end{enumerate}

Beside these, let us also mention, in a slightly different direction, the approach of Agboola and Pappas \cite{ap00}, who start directly from a line bundle on an integral model of $C$.

In any case, when $P$ is fixed, the specialisation map $L\mapsto P^*L$ is a group morphism with respect to the group law on $\Pic^0(C)$. In fact, as we shall see, the map $(P,L)\mapsto P^*L$ can be described in terms of Mazur-Tate's class group pairing \cite{mt}.

In \cite{gl}, the first approach (A1) was considered. Using Kummer theory and Hilbert's irreducibility theorem we proved that, if $C$ is a hyperelliptic curve with a rational Weierstrass point and if $L$ is a torsion line bundle on $C$, then there exist infinitely many quadratic points $P$ on $C$ such that $P^*L$ is a nontrivial ideal class in the $S$-class group of $\Q(P)$, where $S$ is the set of places of bad reduction of $C$. Thus, Question~\ref{mainQ} has a positive answer in this case.

On the other hand, Soleng \cite{soleng94}  proved the following: if $C$ is an elliptic curve, and if $L$ is a line bundle of infinite order with good reduction properties (we shall make this precise later), then one can find a sequence of points $P_n$ such that $\Q(P_n)$ is imaginary quadratic and the order of $P_n^*L$ in the class group of $\Q(P_n)$ is unbounded when $n\to-\infty$. Therefore, Question~\ref{mainQ} is completely settled in the case when $C$ is an elliptic curve.

Let us underline the fact that, despite its strength, Soleng's work came to our knowledge only very recently. Soleng's construction is expressed in the language of quadratic forms, and relies on the choice of a Weierstrass equation for the elliptic curve, which seems a bit \emph{ad hoc}. Nevertheless, one can prove that Soleng's morphism agrees (up to some normalization process) with the specialisation map $L\mapsto P^*L$.

The aim of the present paper is twofold:
\begin{enumerate}
\item[---] extend Soleng's result to hyperelliptic curves whose Weierstrass points are defined over extensions of $\Q$ of degree at most three, and at least one of them is rational (see Theorem~\ref{thm2}); this yields a positive answer to Question~\ref{mainQ} for such curves;
\item[---] relate Soleng's class group morphism and Mazur-Tate's class group pairing (see Section~\ref{sec:BS}); this extends previous comparison results \cite{bc16} to the case of hyperelliptic curves, and gives a more canonical explanation for these phenomena.
\end{enumerate}

Let us now describe in detail the second approach (A2) for defining $P^*L$. Let us assume once and for all that $C(\Q)\neq\emptyset$, which implies that $\Pic^0(C)\simeq J(\Q)$.

Let $\X\to\Spec(\Z)$ be a regular projective flat model of $C$ (we do not require that $\X$ is the minimal model of $C$). 
Let $\Pic_{\X/\Z}$ be the relative Picard functor of $\X\to\Spec(\Z)$. Let us recall \cite[\S{}8.1]{NeronModels} that, given a section $\varepsilon\in\X(\Z)$, one has an isomorphism, functorial in $T$
$$
\Pic_{\X/\Z}(T)\simeq \Pic^r(\X\times T,\varepsilon_T)
$$
where $\Pic^r(\X\times T,s_T)$ denotes the group of rigidified line bundles with respect to $\varepsilon_T:T\to \X\times T$, that is, the group of line bundles on $\X\times T$ together with a trivialisation of $\varepsilon_T^*\mathcal{L}$. The ring $\Z$ being principal, all line bundles over $\X$ can be rigidified. Nevertheless, when working over the ring of integers of an arbitrary number field, one has to take into account rigidifications, which allow to forget the contribution of the Picard group of the base.

Let $\Pic_{\X/\Z}^0$ be the subfunctor of $\Pic_{\X/\Z}$ corresponding to rigidified line bundles on $\X$ which have (partial) degree zero when restricted to each irreducible component of each vertical fiber of $\X\to \Spec(\Z)$. We denote by $\Pic^0(\X)$ the sections of $\Pic_{\X/\Z}$ over $\Z$.

According to Raynaud \cite[\S{}9.5, Thm.~4]{NeronModels}, one has a canonical isomorphism
\begin{equation}
\label{eq:raynaud}
\Pic_{\X/\Z}^0\simeq \J^0
\end{equation}
where $\J$ denotes the N{\'e}ron model of $J$, and $\J^0$ its connected component.

By definition of the relative Picard functor, we have a universal rigidified line bundle $\Poinc_\X$ on $\X\times\Pic^0_{\X/\Z}$, which is usually called the Poincar{\'e} line bundle. By abusing notation, we denote by $\Poinc_\X$ the resulting universal line bundle on $\X\times \J^0$. Then, if $Q\in \J^0(\Z)\subseteq J(\Q)$ is a point, the corresponding line bundle $L_Q\in\Pic^0(C)$ can be uniquely extended into a (rigidified) line bundle $\mathcal{L}_Q$ in $\Pic^0(\X)$. This line bundle $\mathcal{L}_Q$ is given by
\begin{equation}
\label{eq:L_Q}
\mathcal{L}_Q=(\mathrm{id}\times Q)^* \Poinc_\X.
\end{equation}

Consequently, if $P\in C(\overline{\Q})$ is a point, extending into a section $P:\Spec(\mathcal{O}_{\Q(P)})\to\X$, it looks reasonable to define the specialisation of $L_Q$ at $P$ as being
$$
P^*L_Q:=(P\times Q)^* \Poinc_\X
$$
which belongs to $\Pic(\mathcal{O}_{\Q(P)})$.

In fact, this definition does not depend on the choice of the regular model $\X$, more precisely for each $(P,Q)\in C(\overline{\Q})\times \J^0(\Z)$ we have
$$
(P\times Q)^* \Poinc_\X =\langle P,Q \rangle^{\cl}
$$
where $\langle P,Q \rangle^{\cl}$ denotes Mazur-Tate's class group pairing \cite[Remarks~3.5.3]{mt}. We recall the definition of this pairing and prove the equality above in \S{}\ref{subsec:MT}.

We are now ready to state our main results about the non-trivial specialisation of line bundles on hyperelliptic curves.

In our terminology, a hyperelliptic curve over a field $k$ is a smooth projective geometrically connected $k$-curve of genus $g\geq 1$, endowed with a degree $2$ map $C\to \PP^1_k$. This includes elliptic curves over $k$. In this setting, the Weierstrass points of $C$ are just the ramification points of the degree $2$ map $C\to \PP^1_k$.

\begin{thm}
\label{thm1}
Let $C$ be a hyperelliptic curve of genus $g\geq 1$ defined over $\Q$, with a rational Weierstrass point. Let us choose an affine equation of $C$ of the form
\begin{equation}
\label{Weqn}
y^2=f(x)
\end{equation}
where $f\in\Z[x]$ is a square-free monic polynomial of degree $2g+1$. Let $\W\to\Spec(\Z)$ be the integral affine scheme defined by the equation above. For each $n\in\Z$ such that $f(n)<0$, we denote by
$$
T_n:\Spec(\Z[\sqrt{f(n)}]) \longrightarrow \W
$$
the canonical section (\emph{i.e.} the closed subscheme of $\W$ defined by the equation $x=n$).

Then for any $\mathcal{L}\in\Pic(\W)$ which is non-trivial, there exist infinitely many $n\in\Z_{<0}$ such that $T_n^*\mathcal{L}\neq 0$. In particular, if $\mathcal{L}$ has infinite order then the order of $T_n^*\mathcal{L}$ is unbounded when $n\to -\infty$.
\end{thm}

The polynomial $f$ being monic of odd degree, we have $\lim_{x\to-\infty} f(x)=-\infty$, hence there exists an integer $n_f$ such that $f(n)<0$ for all $n\leq n_f$. Whenever $f(n)<0$, the ring $\Z[\sqrt{f(n)}]$ is an order of discriminant $4f(n)$ in the imaginary quadratic field $\Q(P_n)=\Q(\sqrt{f(n)})$.

Let us denote by $\infty$ the unique point at infinity on $C$. The Picard group of $\W$ can be identified with a finite index subgroup of $\Pic(C\setminus\{\infty\})\simeq \Pic^0(C)$ (see Lemma~\ref{restriction}). Therefore, given $Q\in\J^0(\Z)$, up to replacing $Q$ by some multiple, the restriction of $L_Q$ to $C\setminus\{\infty\}$, that we denote by $N_Q$, extends uniquely into a line bundle $\mathcal{N}_Q$ on $\Pic(\W)$. It follows from Corollary~\ref{C:pairing} that the image of $T_n^*\mathcal{N}_Q$ by the normalisation map
$$
\Pic(\Z[\sqrt{f(n)}])\longrightarrow \Pic(\mathcal{O}_{\Q\left(\sqrt{f(n)}\right)})
$$
is equal to $\langle P_n,Q\rangle^{\cl}$, where $P_n$ denotes the generic fiber of the section $T_n$. The size of the kernel of the map above is related to the conductor of $\Z[\sqrt{f(n)}]$, or, equivalently, to the size of the largest square factor of $f(n)$. By controlling the size of this factor, we are able to deduce from Theorem~\ref{thm1} the following result.

\begin{thm}
\label{thm2}
With the hypotheses and notation of Theorem~\ref{thm1}, let $P_n\in C(\overline{\Q})$ be the generic fiber of the section $T_n$.

Assume in addition that one of the following conditions is satisfied:
\begin{enumerate}
\item[(i)] the curve $C$ admits an affine equation of the form \eqref{Weqn} where $f\in\Z[x]$ is a square-free monic polynomial whose irreducible factors have degree at most three;
\item[(ii)] the abc conjecture holds.
\end{enumerate}
Then, given $Q\in \J^0(\Z)$ of infinite order, the order of $\langle P_n,Q\rangle^{\cl}$ is unbounded when $n\to -\infty$.
\end{thm}

Geometrically speaking, condition (i) means that the Weierstrass points of $C$ are defined over extensions of $\Q$ of degree at most three.

In case (i), the proof follows from results of Booker and Browning \cite{BB16}, generalizing results of Hooley \cite{hooley68} on square-free values of polynomials of degree at most three. According to Granville \cite{granville98}, if the abc conjecture holds true, then these results can be generalized to arbitrary square-free polynomials, from which case (ii) follows.

\begin{rmq}
\begin{enumerate}
\item[a)] The results above could probably be generalized when replacing the field $\Q$ by an arbitrary number field, at the price of bigger technicalities. However, it is not clear how one should adapt our proof to this situation.
\item[b)] A natural question arises: let $r$ be the rank of $J(\Q)=\Pic^0(C)$, and let $Q_1,\dots,Q_r$ be a basis of the free part of $J(\Q)$. Is it true that, for a positive density of $n$, the pairings $\langle P_n,Q_1\rangle^{\cl},\dots, \langle P_n,Q_r\rangle^{\cl}$ generate $r$ distinct cyclic factors in the ideal class group of $\Q(\sqrt{f(n)})$? Although we expect a positive answer to this question, we have not been able to handle it with the techniques used in this paper.
\item[c)] The proof of Theorem~\ref{thm1} relies on the classical description of ideals in quadratic rings in terms of binary quadratic forms, and its generalisation by Wood \cite{wood11} to double covers of schemes. It seems tempting to extend the result above to trigonal curves, using the description of ideals in cubic rings discovered by Bhargava \cite{bhargavaII}, which has also been extended to triple covers of schemes by Wood \cite{wood14}.
\end{enumerate}
\end{rmq}

In fact, Theorem~\ref{thm2} is a result about values of the class group pairing for Jacobians of hyperelliptic curves. It is quite natural to extend it to abelian varieties which are isogenous to a product of such Jacobians.

\begin{cor}
\label{cor1}
Let $V$ be an abelian variety over $\Q$ which is isogenous to a product of Jacobians of hyperelliptic curves satisfying the hypotheses of Theorem~\ref{thm2}, and let $\V^t$ be the N{\'e}ron model of its dual abelian variety. Then, given $Q\in \V^{t,0}(\Z)$ of infinite order, there exists a point $P$ on $V$, defined over an imaginary quadratic field, such that $\langle P,Q\rangle^{\cl}\neq 0$. In particular, the order of $\langle P,Q\rangle^{\cl}$ is unbounded when $P$ runs through quadratic points of $V$.
\end{cor}

Let us detail the contents of this paper: in Section~\ref{sec:2}, we prove properties of the model $\W$ defined in Theorem~\ref{thm1}, and we deduce how one can compute values of the specialisation map by restricting line bundles on $\W$. In Section~\ref{sec:Quadratic}, using the foundational paper of Wood \cite{wood11}, whose main results are recalled, we describe the Picard group of $\W$ in terms of quadratic forms, which yields an integral analogue of Mumford's description of the Jacobians of hyperelliptic curves. In Section~\ref{sec:Main} we prove Theorems \ref{thm1} and \ref{thm2}. Finally, in Section~\ref{sec:BS}, we discuss the relation between Mazur-Tate's class group pairing and the so-called class group morphism defined (independently) by Buell \cite{buell76,buell77} and Soleng \cite{soleng94}.


\subsection*{Acknowledgements}

The author was supported by the CIMI Excellence program while visiting the \emph{Centro di Ricerca Matematica Ennio De Giorgi} during the autumn of 2017. 

Many thanks to C{\'e}dric P{\'e}pin for nice email exchange about line bundles on integral models. I also wish to thank Tim Browning and his students for answering my questions about square-free values of polynomials, and the Referee for providing valuable advice. Last but not least, I address my warmest thanks to William Dallaporta for his meticulous reading of the manuscript, which lead to many improvements, and for finding a deeply hidden mistake.


\section{Specialisation map and odd degree Weierstrass models}
\label{sec:2}


\subsection{Basic properties of odd degree Weierstrass models}

Let $C$ be a hyperelliptic curve of genus $g\geq 1$ defined over $\Q$, with a rational Weierstrass point. As recalled in the statement of Theorem~\ref{thm1}, it is well known that such a curve can be defined by an equation of the form \eqref{Weqn}, namely
$$
y^2=f(x)
$$
where $f\in\Z[x]$ is a square-free monic polynomial of degree $2g+1$. In such a model, the curve $C$ has a unique point at infinity, that we shall denote by $\infty$.

Once and for all, we fix an equation \eqref{Weqn}, and we let $\W\to\Spec(\Z)$ be the affine scheme defined by this equation. The generic fiber of $\W$ is the smooth affine curve $C\setminus\{\infty\}$.

Let us denote by $\disc(f)$ the discriminant of $f$. It is straightforward to check that, if $p$ does not divide $2\cdot\disc(f)$, the fiber of $\W$ at $p$ is a smooth, geometrically connected, affine curve over $\mathbb{F}_p$. We shall now describe the singular fibers.

\begin{lem}
\label{Wlemma}
\begin{enumerate}
\item[1)] the map $\W\to\Spec(\Z)$ is flat, and the $x$-coordinate map $x:\W\to\A^1_{\Z}$ is finite flat of rank $2$;
\item[2)] the fibers of $\W\to\Spec(\Z)$ are geometrically integral curves;
\item[3)] $\W$ is normal; in particular, it is regular in codimension 1, and has only finitely many singular points.
\end{enumerate}
\end{lem}

The proof relies on the following Lemma \cite[Lemme~2]{liu96}.

\begin{lem}
\label{liu}
Let $R$ be a discrete valuation ring, with perfect residue field $k$, and let $f\in R[x]$ be a polynomial. We let
$$
W:=\Spec(R[x,y]/(y^2-f)).
$$
Assume that the generic fiber of $W\to\Spec(R)$ is normal. Then the following holds:
\begin{enumerate}
\item[a)] If $W_k$ is reduced, then $W$ is normal.
\item[b)] The scheme $W_k$ is not reduced if and only if $4\tilde{f}=0$ and $-\tilde{f}$ is a square in $k[x]$ (where $\tilde{f}$ denotes the image of $f$ in $k[x]$).
\end{enumerate}
\end{lem}

\begin{proof}[Proof of Lemma~\ref{Wlemma}]
1) is obvious by construction of $\W$. In order to prove 2), let us choose a prime number $p$. Then the special fiber of $\W$ at $p$ is the affine $\mathbb{F}_p$-curve defined by the equation $y^2=f(x)$. Let us note that $f$ is monic of degree $2g+1$ over $\Z$, hence the reduction of $f$ modulo $p$ is again a monic polynomial of degree $2g+1$. By obvious considerations on the degrees, it's not possible to find a polynomial $h\in \overline{\mathbb{F}}_p[x]$ such that $f=h^2$, hence the fiber of $\W$ at $p$ is geometrically irreducible. Let us now check that the fiber of $\W$ at $p$ is geometrically reduced. According to Lemma~\ref{liu} b), this is true if $p\neq 2$. Moreover, the fiber at $p=2$ is geometrically reduced if and only if $-f$ is not a square in $\overline{\mathbb{F}}_2[x]$, which is true again because the reduction of $-f$ mod $2$ has degree $2g+1$.

Finally, we note that the generic fiber of $\W$ is the smooth affine curve $C\setminus\{\infty\}$ over $\Q$, hence is regular. According to 2), the fibers of $\W$ are reduced, hence by Lemma~\ref{liu} a) the scheme $\W$ is normal, which concludes the proof.
\end{proof}

\begin{rmq}
Let us point out that the fiber of $\W$ at $2$ is always singular, because in characteristic $2$ any point $(x,y)$ on the affine curve $y^2=f(x)$ such that $\partial f(x)=0$ is a singular point. Nevertheless, the polynomial $\partial f$ is not identically zero over $\mathbb{F}_2$, because it can be written as
$$
\partial f(x)=(2g+1)x^{2g}+\cdots \equiv x^{2g}+\cdots \pmod{2},
$$
so the fiber of $\W$ at $2$ has finitely many singular points. In case the curve $C$ has good reduction at $2$, there exist an integral model of $C$, defined by an equation of the form $y^2+g(x)y=f(x)$, whose fiber at $2$ is not singular. In order to avoid unnecessary technicalities, we shall not consider such equations.
\end{rmq}

We shall now compare the Picard groups of $C$ and $\W$.

\begin{lem}
\label{restriction}
The restriction to the generic fiber map
\begin{equation}
\label{resmap}
\Pic(\W)\longrightarrow \Pic(C\setminus\{\infty\})
\end{equation}
is injective, with finite cokernel. Moreover, given a Cartier divisor $D$ on $C\setminus\{\infty\}$ whose class belongs to the image of this map, the scheme  theoretic closure of $D$ in $\W$ is a Cartier divisor extending $D$.
\end{lem}

\begin{proof}
Roughly speaking, the injectivity follows from the fact that $\W$ is normal, with geometrically integral fibers. More precisely, let $\Sigma$ be the (finite) set of singular points of $\W$. Because $\W\setminus \Sigma$ is regular, any Cartier divisor on $C\setminus\{\infty\}$ can be extended (by scheme theoretic closure) into a Cartier divisor on $\W\setminus \Sigma$. Therefore, the map 
\begin{equation}
\label{res1}
\Pic(\W\setminus \Sigma)\longrightarrow \Pic(C\setminus\{\infty\})
\end{equation}
is surjective. Its kernel is the subgroup generated by fibral divisors. According to Lemma~\ref{Wlemma}, the fibers of $\W\setminus \Sigma\to\Spec(\Z)$ are geometrically integral, hence any fibral divisor is a sum of fibers. Such a divisor is principal, because $\Z$ is principal. This proves that the map \eqref{res1} is bijective.

The scheme $\W$ being normal, the restriction map $\Pic(\W)\to \Pic(\W\setminus \Sigma)$ is injective, hence the map \eqref{resmap}, which is obtained by composing it with the bijection \eqref{res1},  is injective. It also follows from the normality of $\W$ that, given a Cartier divisor $D$ on $C\setminus\{\infty\}$ whose class belongs to the image of the map \eqref{resmap}, the scheme theoretic closure of $D$ in $\W$ is a Cartier divisor extending $D$.

It remains to prove that the cokernel of the map $\Pic(\W)\to \Pic(\W\setminus \Sigma)$ is finite. If $s$ is a singular point of $\W$,  we let
$$
U_s:= \Spec(\mathcal{O}_{\W,s})\setminus \{s\} \quad\text{(punctured neighbourhood of $s$ in $\W$)}
$$
which is a regular scheme, because $\W$ is regular in codimension 1. We know that $\mathcal{O}_{\W,s}$ is a Noetherian excellent local ring of dimension $2$, and that the residue field of $s$ is finite. Therefore, it follows from \cite[Th{\'e}or{\`e}me~2.8]{lmb89} that $\Pic(U_s)$ is a finite group.

Let us now consider the map
$$
\Pic(\W\setminus \Sigma) \longrightarrow \prod_{s\in \Sigma} \Pic(U_s).
$$

Each factor of the product being finite, the group on the right-hand side is finite. On the other hand, if $\mathcal{L}$ belongs to the kernel of this map, then $\mathcal{L}$ can be extended into a line bundle over $\W$, because the trivial line bundle on $U_s$ extends into the trivial line bundle on $\Spec(\mathcal{O}_{\W,s})$. This means that the kernel of the map above is contained in the image of the map $\Pic(\W)\to \Pic(\W\setminus \Sigma)$, from which the result follows.
\end{proof}

Let us now consider the scheme $\overline\W$ obtained by glueing the affine schemes defined by the equations $y^2=f(x)$ and $v^2=u^{2g+2}f(1/u)$ along the isomorphism $(u,v)=(1/x,y/x^{g+1})$ over the loci where $x\neq 0$ and $u\neq 0$, respectively. Then the $x$-coordinate map $\overline\W\to\mathbb{P}^1_\Z$ is finite flat of rank $2$; a finite morphism being projective, the structural morphism $\overline\W\to \Spec(\Z)$ is projective as composition of projective morphisms. Moreover, $\overline{\W}$ has an integral section corresponding to the point at infinity on $C$, that we still denote by $\infty$. This section at infinity belongs to the smooth locus of $\overline{\W}\to \Spec(\Z)$, and one can write $\overline{\W}=\W\cup\{\infty\}$.

It follows from Lemma~\ref{Wlemma} that $\overline\W\to\Spec(\Z)$ is a flat normal model of $C$, with geometrically integral fibers.  Therefore, any element in the Picard group of $\overline\W$ can be represented by a horizontal divisor. The degree map, which is well-defined on horizontal Cartier divisors, yields a commutative diagram with exact lines
$$
\begin{CD}
0 @>>> \Pic^0(\overline\W) @>>> \Pic(\overline\W) @>\deg>> \Z @>>> 0 \\
@. @VVV @VVV  @| \\
0 @>>> \Pic^0(C) @>>> \Pic(C) @>\deg>> \Z @>>> 0. \\
\end{CD}
$$

Moreover, the map $n\mapsto n\cdot\infty$ is a section of the degree map and induces an isomorphism
\begin{align*}
\Pic(\W) &\longrightarrow \Pic^0(\overline\W) \\
D &\longmapsto \overline{D}-\deg(\overline{D})\cdot \infty
\end{align*}
where $\overline{D}$ denotes the closure of $D$ in $\overline\W$. This fits into a commutative diagram
\begin{equation}
\label{diag2pic}
\begin{CD}
\Pic(\W) @>\sim>> \Pic^0(\overline\W) \\
@VVV @VVV \\
\Pic(C\setminus\{\infty\}) @>\sim>> \Pic^0(C). \\
\end{CD}
\end{equation}

This being said, it follows from Lemma~\ref{restriction} that the restriction map $\Pic^0(\overline\W)\to \Pic^0(C)$ is injective, with finite cokernel. 

We shall now relate degree zero line bundles on $\overline\W$ and on its minimal desingularisation.

\begin{lem}
\label{L:ConCom}
Let $h:\X\to\overline{\W}$ be the minimal desingularisation of $\overline{\W}$. Let $Q\in J(\Q)$ and let $L_Q\in\Pic^0(C)$ be the corresponding degree zero line bundle on $C$. Assume that $L_Q$ extends into a line bundle $\mathcal{M}_Q\in\Pic^0(\overline\W)$. Then $Q$ belongs to $\J^0(\Z)$ and $h^*\mathcal{M}_Q=\mathcal{L}_Q$, where $\mathcal{L}_Q$ is the line bundle defined in the introduction \eqref{eq:L_Q}.
\end{lem}

\begin{proof}
Let us recall that $\overline\W$ is regular outside a finite set $\Sigma$ of closed points of $\overline\W$. According to the moving lemma for Cartier divisors \cite[Prop.~6.2]{GLL13}, if $D$ is a Cartier divisor on $\overline\W$, it is possible to find a Cartier divisor $D'$ on $\overline\W$ which is linearly equivalent to $D$, and whose support is contained in $\overline\W\setminus\Sigma$. Therefore, if $L_Q$ extends into a line bundle $\mathcal{M}_Q\in\Pic^0(\overline\W)$, then $\mathcal{M}_Q$ is the class of some horizontal degree zero divisor $D$ on $\overline\W$ whose support is contained in $\overline\W\setminus\Sigma$.

Let us recall that the minimal desingularisation $h:\X\to\overline{\W}$ is obtained by a finite sequence of blow-ups above singular points of $\overline{\W}$, in particular $h$ induces an isomorphism between $\overline{\W}\setminus \Sigma$ and some dense open subset of $\X$. But the support of $D$ is contained in $\overline{\W}\setminus \Sigma$, hence $h^*D$ coincides with its strict transform. In particular, $h^*D$ is a horizontal Cartier divisor on $\X$, and, for each prime $p$, the restriction of $h^*D$ to $\X_{\mathbb{F}_p}$ is supported in the strict transform of $\overline\W_{\mathbb{F}_p}$, which is an irreducible (and reduced) component of $\X_{\mathbb{F}_p}$, because $\overline\W_{\mathbb{F}_p}$ is integral. We note that the strict transform of $\overline\W_{\mathbb{F}_p}$ is also the component of $\X_{\mathbb{F}_p}$ in which the point at infinity specialises.

It follows from the discussion that the restriction of $h^*D$ to $\X_{\mathbb{F}_p}$ is supported on a single irreducible component of $\X_{\mathbb{F}_p}$, and has total degree zero. Therefore, $h^*D$ has partial degree zero on each of the irreducible components of the vertical fibers of $\X$. Thus $h^*\mathcal{M}_Q$, which is the class of $h^*D$, belongs to $\Pic^0(\X)$. But there exists at most one such line bundle on $\X$ extending $L_Q$, hence $h^*\mathcal{M}_Q=\mathcal{L}_Q$. This also proves that $Q$ belongs to $\J^0(\Z)$, via the isomorphism $\Pic^0(\X)\simeq\J^0(\Z)$.
\end{proof}


\subsection{Specialisation map and class group pairing}
\label{subsec:MT}


The class group pairing was defined by Mazur and Tate \cite[Remarks~3.5.3]{mt} via the theory of biextensions. We briefly recall the definition of this pairing in our setting.

Let $K$ be a number field, and let $V$ be an abelian variety defined over $K$. We denote by $V^t$ the dual abelian variety of $V$. We denote by $\V$ and $\V^t$ their respective N{\'e}ron models over $\Spec(\mathcal{O}_K)$, and by $\V^0$ (resp. $\V^{t,0}$) the respective connected components of the identity.

The duality between $V$ and $V^t$ is encoded into the Poincar{\'e} line bundle. According to \cite[expos{\'e}~VIII, Th{\'e}or{\`e}me~7.1~b)]{gro7}, there exists a unique biextension of $(\V^0, \V^t)$ by the multiplicative group extending the Poincar{\'e} line bundle; we denote by $\Poinc^0_\mathrm{left}$ the underlying line bundle on $\V^0\times \V^t$. Similarly, we have a line bundle $\Poinc^0_\mathrm{right}$ on $\V\times \V^{t,0}$. Let us note that $\Poinc^0_\mathrm{left}$ and $\Poinc^0_\mathrm{right}$ agree when restricted to $\V^0\times\V^{t,0}$; we call them Poincar{\'e} line bundles by abuse of notation.

Let $S$ be the set of places of bad reduction of $\V$, and let us choose a decomposition $S=S_1\sqcup  S_2$ in two disjoint subsets. Let $\V^1$  be the unique dense open subgroup scheme of $\V$  such that $\V^1_v=\V_v$ for each $v\in S_1$, and whose all other fibers are connected (in particular, $\V^1_v=\V^0_v$ for each $v\in S_2$); let $\V^{t,2}$ be the analoguous subscheme of $\V^t$ with respect to the set $S_2$. Then there exists a canonical line bundle (with underlying biextension structure) $\Poinc^{1,2}$ on $\V^1\times \V^{t,2}$, which is obtained by glueing along $\V^0\times\V^{t,0}$ the local Poincar{\'e} line bundle $\Poinc^0_\mathrm{right}$ above $v\in S_1$ together with its analogue $\Poinc^0_\mathrm{left}$ above $v\in S_2$.

Now, given $(P,Q)\in \V^1(\mathcal{O}_K) \times \V^{t,2}(\mathcal{O}_K)$, we are able to define the class group pairing of $P$ and $Q$ by letting
$$
\langle P,Q \rangle^{\cl}:= (P\times Q)^*\Poinc^{1,2}.
$$

It follows immediately from the existence of a biextension structure underlying $\Poinc^{1,2}$ that this pairing is bilinear, with values in $\Pic(\mathcal{O}_K)$.

In general, we say that $(P,Q)\in V(K)\times V^t(K)$ is a \emph{good reduction pair} if there exists a decomposition $S=S_1\sqcup S_2$ such that $(P,Q)$ defines a section of $\V^1\times \V^{t,2}$. Equivalently, for each finite place $v$ of $K$, $P$ belongs to $\V^0_v$ or $Q$ belongs to $\V^{t,0}_v$.

We shall now relate values of the class group pairing on the Jacobian with our specialisation map. We consider, as in the beginning of the introduction, a smooth projective geometrically connected curve $C$ defined over $\Q$, such that $C(\Q)\neq\emptyset$. Let $J$ be the Jacobian of $C$, and let $\J_{/K}$ be the N{\'e}ron model of $J$ over $\Spec(\mathcal{O}_K)$. Let $\iota:C\to J$ be the embedding of $C$ in its Jacobian induced by some $\Q$-rational point of $C$. Then $\iota$ gives rise to an isomorphism between $J$ and its dual abelian variety, which itself is encoded by the Poincar{\'e} line bundle $\Poinc$ on $J\times J$. As previously, we denote by $\Poinc^0_\mathrm{left}$ (resp. $\Poinc^0_\mathrm{right}$) its unique extension on $\J_{/K}^0\times \J_{/K}$ (resp. $\J_{/K}\times \J_{/K}^0$) admitting a biextension structure. We define similarly $\Poinc^{1,2}$ for each choice of a decomposition of the set $S$ of bad places of $\J_{/K}$ into $S_1\sqcup S_2$.

By a slight abuse of notation, we say that $(P,Q)\in C(K)\times J(K)$ is a \emph{good reduction pair} if $(\iota(P),Q)$ is a good reduction pair, in which case we let
$$
\langle P,Q \rangle^{\cl}:=\langle \iota(P),Q \rangle^{\cl}=(\iota(P)\times Q)^*\Poinc^{1,2},
$$
which does not depend on the choice of the embedding $\iota:C\to J$; indeed, the underlying structure of biextension on $\Poinc^{1,2}$ implies its invariance by translation.

We shall now prove that, in the case when $Q$ is a section of $\J^0$, the class group pairing can be computed directly on a regular model of $C$, and coincides with our specialisation map, as claimed in the introduction.

\begin{lem}
\label{lem:classOK}
\begin{enumerate}
\item[1)] Let $\X_{/K}$ be a regular projective flat model of $C$ over $\Spec(\mathcal{O}_K)$.  Let $\Poinc_{\X/K}$ be the universal Poincar{\'e} line bundle on $\X_{/K}\times \J_{/K}^0$, and let $\Poinc_{\X/K}^\mathrm{sm}$ be its restriction to $\X_{/K}^\mathrm{sm}\times \J_{/K}^0$, where $\X_{/K}^\mathrm{sm}$ denotes the smooth locus of $\X_{/K}$. Let
$$
\iota^\mathrm{sm}:\X_{/K}^\mathrm{sm}\to \J_{/K}
$$
be the map\footnote{Whose existence and unicity is ensured by the universal property of $\J_{/K}$.} extending $\iota$. Then we have
\begin{equation}
\label{eq:poincOK}
\Poinc_{\X/K}^\mathrm{sm}=(\iota^\mathrm{sm} \times \mathrm{id})^*\Poinc^0_\mathrm{right}.
\end{equation}
\item[2)] Let $(P,Q)\in C(K)\times J(K)$ such that $Q\in \J^0(\mathcal{O}_K)$, and let $\X$ be a regular projective flat model of $C$ over $\Spec(\Z)$. Then
$$
(P\times Q)^* \Poinc_\X = \langle P,Q \rangle^{\cl}.
$$
\end{enumerate}
\end{lem}

\begin{proof}
1) The underlying biextension structure on $\Poinc^0_\mathrm{right}$ induces on $(\iota^\mathrm{sm} \times \mathrm{id})^*\Poinc^0_\mathrm{right}$ a structure of extension of $\J_{/K}^0$ by the multiplicative group $\Gm$. On the other hand, the restriction morphism
$$
\Ext^1_{\X_{/K}}(\J_{/K}^0,\Gm) \to \Ext^1_{\X_{/K}^\mathrm{sm}}(\J_{/K}^0,\Gm)
$$
is an isomorphism, because $\J_{/K}^0 $ has connected fibers, $\X_{/K}$ is regular and $\X_{/K}^\mathrm{sm}$ is a dense open subset \cite[expos{\'e}~VIII, Th{\'e}or{\`e}me~7.1 and Remarque~7.2]{gro7}. In other terms, there exists a unique line bundle on $\X_{/K}\times \J_{/K}^0$ extending $(\iota^\mathrm{sm} \times \mathrm{id})^*\Poinc^0_\mathrm{right}$ and which has an underlying structure of extension; we denote it by $\mathcal{E}$.

According to \cite[expos{\'e}~VII, 1.1.6]{gro7}, there is a canonical isomorphism of line bundles on $\X_{/K}\times \J_{/K}^0\times \J_{/K}^0$
$$
m^*\mathcal{E} \simeq p_1^*\mathcal{E} \otimes p_2^*\mathcal{E}
$$
where $m$ is the group law on $\J_{/K}^0$, and the $p_i$ are the projections; in other terms, $\mathcal{E}$ satisfies the theorem of the square. It follows that the map
\begin{align*}
\J_{/K}^0 &\longrightarrow \Pic_{\X/\mathcal{O}_K} \\
(t:T\to \J_{/K}^0) &\longmapsto (\mathrm{id}\times t)^*\mathcal{E}
\end{align*}
is a morphism of group functors. Because $\J_{/K}^0 $ has connected fibers, this morphism factors through the connected component of the Picard functor, which is none other than $\Pic_{\X/\mathcal{O}_K}^0$. This proves that $\mathcal{E}$ is a section of $\Pic_{\X/\mathcal{O}_K}^0$, hence is the unique rigidified line bundle with partial degree zero extending its generic fiber $(\iota\times \mathrm{id})^*\Poinc$. On the other hand, $\Poinc_{\X/K}$ is the unique such line bundle extending $\Poinc_{C/K}$, and $\Poinc_{C/K}=(\iota\times \mathrm{id})^*\Poinc_K$ by construction. We deduce that $\Poinc_{\X/K}=\mathcal{E}$, from which \eqref{eq:poincOK} follows by restricting these line bundles to $\X_{/K}^\mathrm{sm}\times \J_{/K}^0$.

2) Let us recall that $\mathcal{L}_Q=(\mathrm{id}\times Q)^*\Poinc_\X$ is the unique rigidified line bundle with partial degree zero on $\X\times\Spec(\mathcal{O}_K)$ extending $L_Q$. The scheme $\X\times\Spec(\mathcal{O}_K)$ being not regular in general, we denote by
$$
g:\X_{/K}\rightarrow \X\times \Spec(\mathcal{O}_K)
$$
its minimal desingularisation, which is a projective flat model of $C$ over $\Spec(\mathcal{O}_K)$. Then $g^*\mathcal{L}_Q=(g\times Q)^*\Poinc_\X$ is a rigidified line bundle with partial degree zero on $\X_{/K}$ extending $L_Q$. By unicity of such a line bundle, we have
$$
(g\times Q)^*\Poinc_\X=(\mathrm{id}\times Q)^*\Poinc_{\X/K},
$$
hence
\begin{equation}
\label{eq:poincXK}
(P\times Q)^*\Poinc_\X=(P\times Q)^*\Poinc_{\X/K}.
\end{equation}

On the other hand, the point $P$ belongs to $\X_{/K}^\mathrm{sm}(\mathcal{O}_K)$, because rational sections belong to the smooth locus of a regular scheme \cite[\S{}3.1, Prop.~2]{NeronModels}. Hence it follows from \eqref{eq:poincOK} that
$$
(P\times Q)^*\Poinc_{\X/K}=(\iota(P)\times Q)^*\Poinc^0_\mathrm{right}.
$$
Combining this with \eqref{eq:poincXK} yields
$$
(P\times Q)^* \Poinc_\X= (\iota(P)\times Q)^*\Poinc^0_\mathrm{right}=\langle P,Q \rangle^{\cl}
$$
where the second equality holds by definition of $\langle P,Q \rangle^{\cl}$.
\end{proof}

Now, and till the end of this section, we consider again the hyperelliptic curve $C$ defined by the equation \eqref{Weqn}.
We shall deduce from the previous Lemma that, if a point $P\in C(\overline\Q)$ extends into a section of the affine scheme $\W$, then the class group pairing of $P$ with (suitable) sections of $\J^0$ can be computed by pulling back line bundles on $\W$.

\begin{cor}
\label{C:pairing}
Let $Q\in J(\Q)$ and let $L_Q\in\Pic^0(C)$ be the corresponding degree zero line bundle on $C$. Let $N_Q$ be the restriction of $L_Q$ to $C\setminus\{\infty\}$. Assume that $N_Q$ extends into a line bundle $\mathcal{N}_Q\in\Pic(\W)$. Then $Q$ belongs to $\J^0(\Z)$ and, for any number field $K$ and any $P\in \W(\mathcal{O}_{K})$, we have
$$
\langle P,Q\rangle^{\cl}=P^*\mathcal{N}_Q.
$$
\end{cor}

\begin{proof}
According to diagram \eqref{diag2pic}, $L_Q$ extends into a degree zero line bundle $\mathcal{M}_Q\in\Pic^0(\overline{\W})$ if and only if $N_Q$ extends into a line bundle $\mathcal{N}_Q\in\Pic(\W)$. By Lemma~\ref{L:ConCom}, this extension property implies that $Q$ belongs to $\J^0(\Z)$.

Let $K$ be a number field and let $P\in C(K)$ be a point. By abuse of notation, we also denote by $P:\Spec(\mathcal{O}_K)\to\overline{\W}$ the section extending $P$. Then it follows from Lemma~\ref{L:ConCom} that
$$
P^*\mathcal{M}_Q=P^*\mathcal{L}_Q=(P\times Q)^*\Poinc_\X
$$
and this quantity is equal to $\langle P,Q\rangle^{\cl}$ according to Lemma~\ref{lem:classOK}, 2).

Moreover, by the very definition of the isomorphism $\Pic^0(\overline{\W})\simeq\Pic(\W)$, the line bundle $\mathcal{N}_Q$ is the restriction of $\mathcal{M}_Q$ to $\W$. Therefore, in the case where $P$ belongs to $\W(\mathcal{O}_{K})$ (\emph{i.e.} $P$ is a section of the affine scheme $\W$), we have
$$
P^*\mathcal{M}_Q=P^*\mathcal{N}_Q,
$$
hence the result.
\end{proof}


\section{Quadratic forms and Picard groups of double covers}
\label{sec:Quadratic}


\subsection{Wood's result}

It has been proved in great generality by M. Wood in her foundational paper \cite{wood11} that invertible modules over double covers of schemes can be described by binary quadratic forms. We shall now specialise her definitions and results to our situation.

Alternatively, the material explained in this section could also be extracted, with a bit of effort, from earlier work of Kneser \cite{kneser}.

\begin{dfn}
\label{Qdfn}
Let $R$ be a ring such that every locally free $R$-module of finite rank is free, and let $D\neq 0$ be an element of $R$.
\begin{enumerate}
\item A \emph{binary quadratic form} over $R$, with discriminant $4D$, is an expression of the form
$$
F=aX^2+2bXY+cY^2
$$
where $a$, $b$ and $c$ belong to $R$, and $b^2-ac=D$.
\item We denote the form above as $F=[a,2b,c]$.
\item We say that the form $F=[a,2b,c]$ is \emph{primitive} if $(a,2b,c)=(1)$ as an ideal of $R$.
\item We let $\GL_2(R)\times \Gm(R)$ act on quadratic forms as follows :
$$
(g,u).F(X,Y)=uF(g(X,Y)).
$$
We say that two forms are \emph{equivalent} if they belong to the same orbit under this action.
\item We denote by $\PQuad(R,4D)$ the set of equivalence classes of primitive binary quadratic forms over $R$, with discriminant $4D$.
\end{enumerate}
\end{dfn}

\begin{rmq}
Let us note that two equivalent forms have the same discriminant up to the square of a unit of $R$. Nevertheless, if we rescale the coefficients by the inverse of this unit via the action of $\Gm(R)$, this allows to remove the square factor.

For example, let $[a,2b,c]$ be a quadratic form over $R$, and let $\lambda\in R^\times$. If we let the matrix
$\left(\begin{smallmatrix}
    \lambda & 0\\
    0 & 1
\end{smallmatrix} \right)$
act on this form, we get
$$
[a,2b,c]\sim [\lambda^2a,2\lambda b,c]
$$
and the discriminant is multiplied by $\lambda^2$. If we let the couple
$(\left(\begin{smallmatrix}
    \lambda & 0\\
    0 & 1
\end{smallmatrix} \right),\lambda^{-1})$
act on the same form, we get
\begin{equation}
\label{eq_scaling}
[a,2b,c]\sim [\lambda a,2b,\lambda^{-1} c]
\end{equation}
and both forms have the same discriminant.
\end{rmq}

Let $F=[a,2b,c]$ be a binary quadratic form over $R$, with discriminant $4D$. We can associate to $F$ an ideal $I_F$ of the quadratic algebra $R[y]/(y^2-D)$ defined as follows:
$$
I_F:=(a,y-b)=aR\oplus (y-b)R.
$$

The fact that $I_F$ is an ideal follows from the fact that the relation $\overline{b}^2=\overline{D}$ holds in the quotient ring $R/aR$. Indeed, under this assumption, the map
\begin{align*}
R[y]/(y^2-D) &\longrightarrow R/aR \\
r+sy &\longmapsto \overline{r+sb}
\end{align*}
is a surjective morphism of algebras, with kernel $I_F$. Conversely, given two elements $a$ and $b$ in $R$, if the $R$-module $aR\oplus (y-b)R$ is an ideal of $R[y]/(y^2-D)$, then the relation $\overline{b}^2=\overline{D}$ holds in the quotient ring $R/aR$, hence there exists $c\in R$ such that $b^2-D=ac$.

The following result follows from Wood's constructions. It is a generalisation of the classical description of ideal class groups of quadratic fields in terms of quadratic forms. As we shall see later, it also yields an integral version of Mumford's representation of divisor classes on hyperelliptic curves.

\begin{thm}
\label{quadrep}
Let $R$ and $D$ be as above. Then the map
\begin{align*}
\PQuad(R,4D) &\longrightarrow \Pic(R[y]/(y^2-D)) \\
F=[a,2b,c] &\longmapsto I_F=(a,y-b)
\end{align*}
is well-defined, and bijective.
\end{thm}

\begin{proof}
Let us prove that the map above is well-defined. We need to check that, if $(a,2b,c)=(1)$, then $(a,y-b)$ is an invertible ideal of the quadratic ring $R[y]/(y^2-D)$.

We have:
$$
(y-b)(y+b)=y^2-b^2=D-b^2=-ac
$$
from which it follows that
\begin{align*}
(a,y-b)(a,y+b) &=(a^2,a(y-b),a(y+b),ac) \\
 &= (a)(a,y-b,y+b,c).
\end{align*}
Let us note that $(1)=(a,2b,c)\subseteq (a,y-b,y+b,c)$ hence the ideal on the right is equal to $(1)$. Therefore,
$$
(a,y-b)(a,y+b)=(a)
$$
which proves that $(a,y-b)$ is an invertible ideal, and its inverse is $(a,y+b)$.

We shall now explain how the bijectivity of our map is a special case of Wood's result. What Wood calls a linear binary quadratic form over some scheme $X$ is the data of a locally free rank $2$ $\mathcal{O}_X$-module $V$, a locally free rank $1$ $\mathcal{O}_X$-module $L$, and a global section of $\mathrm{Sym}^2V\otimes L$. In our setting, every locally free $R$-module of finite rank is free by assumption. One checks that linear binary quadratic forms over $\Spec(R)$ in the sense of Wood correspond exactly to binary quadratic forms over $R$ as in our Definition~\ref{Qdfn}, with the same concept of primitive form and of equivalence between two forms. The desired bijection now follows from \cite[Theorem~1.5]{wood11}. The explicit description of this bijection follows from a careful check of the local constructions in \cite[Section~2]{wood11}.
\end{proof}


\subsection{The Picard group of $\W$}

Consider again the hyperelliptic curve $C$ defined by the equation \eqref{Weqn}. The ring $\Q[x]$ being a principal ideal domain, it follows from Theorem~\ref{quadrep} that line bundles on $C\setminus\{\infty\}$ can be represented by primitive quadratic forms over $\Q[x]$, with discriminant $4f$. Moreover, $\Q[x]$ being an Euclidean domain, one can generalize the reduction algorithm of Gauss (for quadratic forms over $\Z$) to this setting, as it was noted by Cantor \cite{cantor}. This being done, we recover the description of degree zero divisor classes on $C$ due to Mumford \cite[Chap.~IIIa]{tata2}.

\begin{lem}[Mumford representation]
\label{mumfordrep}
Every element in $\Pic(C\setminus\{\infty\})$ can be uniquely represented by a quadratic form $[a,2b,c]$ over $\Q[x]$, with discriminant $4f$, where:
\begin{enumerate}
\item[(1)] $a$ is monic;
\item[(2)] $\deg b < \deg a \leq g$.
\end{enumerate}
In this correspondence, the quadratic form $F=[a,2b,c]$ corresponds to the divisor
$$
D_F:=\divisor(a)\cap\divisor(y-b)=\sum_{i=1}^r P_i
$$
where $P_i=(x_i,y_i)$, the $x_i$ are the roots of $a$ in $\overline{\Q}$, and $b(x_i)=y_i$.
\end{lem}

Let us note that the degree zero divisor on $C$ corresponding to $[a,2b,c]$ under the isomorphism $\Pic(C\setminus\{\infty\})\simeq \Pic^0(C)$ from diagram \eqref{diag2pic} is
$$
\sum_{i=1}^r P_i -r\cdot \infty.
$$

\begin{rmq}
\label{rmq:squarefree}
Every quadratic form $F=[a,2b,c]$ over $\Q[x]$, with discriminant $4f$, is primitive. Indeed, let $t$ be a point of $\A^1$ such that the fiber of $F$ at $t$ is zero. Then $a$, $b$ and $c$ vanish at $t$, which implies that $f=b^2-ac$ vanishes with order at least two at $t$. This is impossible, because $f$ is square-free by assumption.
\end{rmq}

Let us recall the classical result of Seshadri \cite{seshadri58} : $\Z$ being a principal ideal domain, every finitely generated locally free $\Z[x]$-module is free. Therefore, it follows from Theorem~\ref{quadrep} that line bundles on $\W$ can be represented by primitive quadratic forms over $\Z[x]$.

\begin{lem}[Integral Mumford representation]
\label{intmumford}
Every element of $\Pic(\W)$ can be represented by a primitive quadratic form over $\Z[x]$, with discriminant $4f$.
\end{lem}

Given $\mathcal{L}\in\Pic(\W)$, the primitive quadratic form $[a,2b,c]$ over $\Z[x]$ which represents $\mathcal{L}$ is in general far from the quadratic form over $\Q[x]$ which represents its generic fiber given by Mumford's representation. Roughly speaking, the form over $\Z[x]$ can be obtained from that over $\Q[x]$ by some kind of desingularisation process, which eliminates the denominators. This process has the disastrous effect of enlarging the degrees of the polynomials $a$, $b$ and $c$. See \cite[Example~3.3]{soleng11} for an example involving elliptic curves.

On the other hand, in order to generalise Soleng's proof to our setting, we need to work with a form $[a,2b,c]$ such that $\deg(a)<\deg(f)$. So the integral version of Mumford's representation may be too large for our purpose.

We shall now give an alternative version of Mumford's representation which allows to handle denominators of $[a,2b,c]$ without enlarging the degrees of $a$, $b$ and $c$.

If $h\in\Z[x]$ is a polynomial with integer coefficients, the \emph{content} of $h$ is the gcd of its coefficients; we denote it by $\cont(h)$. We extend this definition to polynomials with rational coefficients by letting $\cont(\lambda h):=\lambda\cont(h)$ for any rational number $\lambda\neq 0$. It's easy to check that the content is multiplicative, and that $\cont(h)$ is an integer if and only if $h$ belongs to $\Z[x]$.

\begin{lem}
\label{altmumford}
Every element in $\Pic(C\setminus\{\infty\})$ can be represented by a quadratic form
$$
\left[\frac{A}{e},2\frac{B}{e},\frac{C}{e}\right]
$$
with discriminant $4f$, where:
\begin{enumerate}
\item[(1)] $A$, $B$ and $C$ belong to $\Z[X]$;
\item[(2)] $\cont(A)=1$;
\item[(3)] $e>0$ in an integer coprime to $\cont(B)$;
\item[(4)] $\deg B < \deg A \leq g$.
\end{enumerate}
Such a representation is unique, up to the sign of the leading coefficient of $A$. Moreover, $[A,2B,C]$ is a quadratic form over $\Z[x]$, with discriminant $4e^2f$.
\end{lem}

\begin{proof}
Let $[a,2b,c]$ be a quadratic form over $\Q[x]$ with discriminant $4f$, in Mumford representation (Lemma~\ref{mumfordrep}). Let $\lambda\in \Q^{\times}$ such that $\lambda a$ belongs to $\Z[x]$ and $\cont(\lambda a)=1$. Then according to \eqref{eq_scaling}, we have
$$
[a,2b,c]\sim_{\Q[x]} [\lambda a,2b,\lambda^{-1}c].
$$

Let us put $A:=\lambda a$. By elementary considerations, there exists $B\in\Z[x]$ and an integer $e>0$ such that
$$
b = \frac{B}{e}
$$
and one may choose them in such a way that $e$ is coprime to $\cont(B)$. Finally we claim that $C:=e^2\lambda^{-1}c$ belongs to $\Z[x]$. Indeed, the relation $b^2-ac=f$ is equivalent to
$$
B^2-AC=e^2f
$$
hence $AC$ belongs to $\Z[x]$. But $\cont(A)=1$, so $\cont(C)=\cont(AC)$ is an integer, which proves that $C$ belongs to $\Z[x]$. Therefore, $[A,2B,C]$ is a quadratic form over $\Z[x]$, with discriminant $4e^2f$, and
$$
[a,2b,c]\sim_{\Q[x]} \left[A,2\frac{B}{e},\frac{C}{e^2}\right]\sim_{\Q[x]} \left[\frac{A}{e},2\frac{B}{e},\frac{C}{e}\right].
$$
The unicity (up to sign) follows from the unicity of the Mumford representation. 
\end{proof}

\begin{lem}
\label{lem:closure}
Let
$$
F=\left[\frac{A}{e},2\frac{B}{e},\frac{C}{e}\right]
$$
be a quadratic form as in Lemma~\ref{altmumford}. Let $\mathcal{Z}_F\subset \W$ be the closed subscheme defined by the ideal
$$
\mathcal{I}_F:=(A,ey-B)
$$
Then the generic fiber of $\mathcal{Z}_F$ is the divisor $D_F$ on $C\setminus\{\infty\}$ corresponding to the quadratic form $F$. We have
\begin{equation}
\label{eq_closure}
\overline{D_F}\subseteq \mathcal{Z}_F
\end{equation}
and these two subschemes agree on the open subset of $\W$ over which $e$ is invertible.
\end{lem}

\begin{proof}
By definition of the scheme theoretic closure, $\overline{D_F}$ is the smallest closed subscheme of $\W$ with generic fiber $D_F$, hence is contained in $\mathcal{Z}_F$, which proves \eqref{eq_closure}. Let us prove that these two subschemes agree on the open subset where $e$ is invertible. We have an isomorphism
$$
\Z\left[\frac{1}{e},x,y\right]/(y^2-f) \simeq \Z\left[\frac{1}{e},x,y\right]/(y^2-e^2f)
$$
defined by $(x,y)\mapsto (x,y/e)$. Hence, it suffices to prove the statement when $e=1$.

Let us assume that $e=1$, \emph{i.e.} $[A,2B,C]$ is a quadratic form over $\Z[x]$, with discriminant $4f$, and $\cont(A)=1$. Then the relations $A=0$ and $y=B$ imply that $y^2=f$. Therefore, the closed subscheme of $\W$ defined by the ideal $\mathcal{I}_F$ is isomorphic to
$$
\Spec(\Z[x,y]/(A,y-B))\simeq \Spec(\Z[x]/(A)),
$$
and this last scheme is the closure in $\W$ of its generic fiber, because $\cont(A)=1$, hence the result.
\end{proof}


\subsection{Remarks on quadratic $\Z$-algebras}

A quadratic $\Z$-algebra is a commutative ring which is free of rank $2$ as a $\Z$-module; it is said to be \emph{non-degenerate} if its tensor product with $\Q$ is a quadratic field. For example, given $D\in\Z$, the quadratic $\Z$-algebra $\Z[y]/(y^2-D)$ is non-degenerate if and only if $D$ is not a square. In this case, we have
$$
\Z[y]/(y^2-D)\simeq\Z[\sqrt{D}].
$$

Throughout this section, we consider a non-degenerate quadratic $\Z$-algebra of the form $\Z[\sqrt{D}]$ and, by abuse of notation, we denote by $y$ the class of $y$ modulo $(y^2-D)$.

The facts stated below are elementary. 
For simplicity, the data of a quadratic form $[a,2b,c]$ over $\Z$ with discriminant $4D$ is replaced by the data of a couple $(a,b)\in\Z^2$ satisfying the condition $a\mid (b^2-D)$, which amounts to the same thing.

\begin{facts}
\label{3facts}
Let us consider a non-degenerate quadratic $\Z$-algebra $\Z[y]/(y^2-D)\simeq\Z[\sqrt{D}]$.
\begin{enumerate}
\item[(i)] If $I\neq (0)$ is an ideal of $\Z[\sqrt{D}]$, there exists $s$, $a$ and $b$ in $\Z$ such that
$$
I=(s)(a,y-b)  \qquad \text{with } a\mid (b^2-D).
$$
Moreover:
\begin{itemize}
\item up to sign, the integers $s$ and $a$ are uniquely determined by $I$
\item if
$$
(s)(a,y-b) \subseteq (t)(u,y-v) \qquad \text{with } a\mid (b^2-D) ~\text{and}~ u\mid (v^2-D)
$$
then $t\mid s$ and $tu\mid sa$.
\end{itemize}
\item[(ii)] If $a$ and $e$ are coprime, and $a\mid (b^2-e^2D)$, then for any $e'$ such that $ee'\equiv 1\pmod{a}$, we have $a\mid ((be')^2-D)$ and
$$
(a,ey-b)\Z[\sqrt{D}]=(a,y-be').
$$
\end{enumerate}
\end{facts}

\begin{proof}
(i) We note that $I$ is a lattice in $\Z[\sqrt{D}]=\Z\oplus y\Z$, hence by the existence of the Hermite normal form of a lattice, there exist integers $s$, $a$ and $b$ such that
$$
I=sa\Z \oplus s(y-b)\Z.
$$

Up to sign, the integer $s$ is the largest positive integer such that $I\subseteq s\Z[\sqrt{D}]$, hence is uniquely determined by $I$. Moreover,  $s^{-1}I=a\Z \oplus (y-b)\Z$ is also an ideal of $\Z[\sqrt{D}]$, hence the relation $a\mid (b^2-D)$ is satisfied according to the discussion before Theorem~\ref{quadrep}.

Finally, $sa\Z = I\cap \Z$, hence $sa$ is (up to sign) uniquely determined by $I$, and so does $a$. The statement on inclusions follows from the same machinery.

(ii) By construction, $e$ and $e'$ are mutual inverses in $\Z/a\Z$, hence the relation $\overline{e}^2\overline{D}=\overline{b}^2$ in $\Z/a\Z$ is equivalent to the relation $\overline{D}=(\overline{be'})^2$.
\end{proof}

Given an ideal $I$ of $\Z[e\sqrt{D}]$ coming from a quadratic form, the following Lemma allows one to describe explicitely $I\Z[\sqrt{D}]$.

\begin{lem}
\label{lemideal}
With the notation above, let $e>0$ be an integer, and let us consider an ideal of $\Z[e\sqrt{D}]$ of the form
$$
I=(a,ey-b) \qquad \text{where } a\mid (b^2-e^2D).
$$
Let $d=\gcd(a,e,b)$. Then there exists $a'$ and $b'$ in $\Z$ such that
$$
I\Z[\sqrt{D}]=(d)(a',y-b')
$$
where $a'\mid (b'^2-D)$ and
$$
a'=\frac{a}{d'} \qquad\text{with } d\mid d'\mid d^2.
$$
\end{lem}

\begin{proof}
Let $d:=\gcd(a,e,b)$. We put $a=da_1$, $e=de_1$ and $b=db_1$. Then $\gcd(a_1,e_1,b_1)=1$ and
$$
I\Z[\sqrt{D}]=(d)(a_1,e_1y-b_1)\Z[\sqrt{D}].
$$

In the ring $\Z[\sqrt{D}]$, we have the relation
$$
(e_1y-b_1)(e_1y+b_1) = e_1^2y^2-b_1^2 = e_1^2D - b_1^2
$$
which proves that $b_1^2-e_1^2D$ belongs to $(a_1,e_1y-b_1)\Z[\sqrt{D}]$. It follows that
\begin{align*}
(a_1,e_1y-b_1)\Z[\sqrt{D}] &=(a_1,b_1^2-e_1^2D,e_1y-b_1)\Z[\sqrt{D}] \\
 &= (\gcd(a_1,b_1^2-e_1^2D),e_1y-b_1)\Z[\sqrt{D}],
\end{align*}
where the last equality follows from B{\'e}zout's identity.

The relation $a\mid (b^2-e^2D)$ implies that
$$
a_1\mid d(b_1^2-e_1^2D),
$$
from which it follows immediately that one can write
$$
\gcd(a_1,b_1^2-e_1^2D) = \frac{a_1}{d_1} \qquad\text{with } d_1\mid d.
$$

Let us put
$$
a':= \gcd(a_1,b_1^2-e_1^2D).
$$
Then $\gcd(a',e_1,b_1)=1$, because $\gcd(a_1,e_1,b_1)=1$. Also, by construction, $a'\mid (b_1^2-e_1^2D)$ so, if some prime $p$ divides $a'$ and $e_1$, then $p$ divides also $b_1$. It follows that $\gcd(a',e_1)=1$. According to Fact~(ii), if we pick $e'$ such that $ee'\equiv 1\pmod{a'}$, then
$$
(a',e_1y-b_1)\Z[\sqrt{D}]=(a',y-b_1e')
$$
and $a'\mid ((b_1e')^2-D)$. To conclude, we have
$$
I\Z[\sqrt{D}]=(d)(a',y-b_1e')
$$
and
$$
a'=\frac{a_1}{d_1} = \frac{a}{d_1d} \qquad\text{with } d_1\mid d.
$$
We get the result by letting $b'=b_1e'$ and $d'=d_1d$.
\end{proof}


\section{Proofs of the main results}
\label{sec:Main}


\subsection{Definition and properties of the specialisation map}

The polynomial $f$ being monic of odd degree, we have
$$
\lim_{n\to -\infty} f(n)= -\infty.
$$

Let $n_f\in \Z$ be the largest integer such that $f(n)<0$ for all $n\leq n_f$. For $n\leq n_f$, $\Z[\sqrt{f(n)}]$ is an order of discriminant $4f(n)$ in the imaginary quadratic field $\Q(\sqrt{f(n)})$. We denote by $T_n:\Spec(\Z[\sqrt{f(n)}])\to \W$ the canonical section.

We note that $T_n$ identifies $\Spec(\Z[\sqrt{f(n)}])$ with the closed subscheme of $\W$ defined by the equation $x=n$.

\begin{lem}
\label{quadlemma}
Let $\mathcal{L}\in\Pic(\W)$, with generic fiber $L\in\Pic(C\setminus \{\infty\})$. Let
$$
F=\left[\frac{A}{e},2\frac{B}{e},\frac{C}{e}\right]
$$
be the quadratic form as in Lemma~\ref{altmumford} whose class represents $L$. Let $n\leq n_f$ be an integer, and let $d(n):=\gcd(A(n),e)$. Then $T_n^*\mathcal{L}$ is the class of a quadratic form of the type
$$
[u(n),2v(n),w(n)],
$$
where $u(n)$ satisfies
$$
\left|\frac{A(n)}{\prod_{p\mid d(n)} p^{v_p(A(n))}}\right| \leq |u(n)| \leq |A(n)|.
$$
\end{lem}

\begin{proof}
Let $D_F$ be the Cartier divisor on $C\setminus\{\infty\}$ attached to $F$, whose class is $L$. According to Lemma~\ref{restriction}, the scheme theoretic closure of $D_F$ in $\W$ is a Cartier divisor and its class is the unique line bundle $\mathcal{L}$ on $\W$ extending $L$. Let $\mathcal{K}_F$ be the ideal sheaf defining the scheme theoretic closure of $D_F$ in $\W$. According to Lemma~\ref{lem:closure}, letting $\mathcal{I}_F:=(A,ey-B)$ we have
$$
\mathcal{I}_F\subseteq \mathcal{K}_F
$$
and these two ideals agree on the open subset of $\W$ over which $e$ is invertible. Therefore, restricting along the section $T_n:\Spec(\Z[\sqrt{f(n)}])\to \W$, we have
\begin{equation}
\label{eq:inclusion}
T_n^*\mathcal{I}_F\subseteq T_n^*\mathcal{K}_F
\end{equation}
and these two ideals agree when restricted to $\Z[\frac{1}{e},\sqrt{f(n)}]$. But, by definition of $\mathcal{I}_F$, one has
$$
T_n^*\mathcal{I}_F=(A(n),e\sqrt{f(n)}-B(n))\Z[\sqrt{f(n)}]
$$
with $A(n)\mid (B^2(n)-e^2f(n))$, because $[A,2B,C]$ is a quadratic form over $\Z[x]$, with discriminant $4e^2f$.
In particular, $T_n^*\mathcal{I}_F$ is the unit ideal when restricted to $\Z[\frac{1}{A(n)},\sqrt{f(n)}]$, hence the same holds for $T_n^*\mathcal{K}_F$, which contains it. One deduces that the two ideals in \eqref{eq:inclusion} agree on $\Z[\frac{1}{d(n)},\sqrt{f(n)}]$.

Let $d:=\gcd(A(n),e,B(n))$, then by Lemma~\ref{lemideal} there exists $a'$ and $b'$ in $\Z$ such that
$$
T_n^*\mathcal{I}_F=(d)(a',\sqrt{f(n)}-b')
$$
where $a'\mid (b'^2-f(n))$ and
\begin{equation}
\label{doud2}
a'=\frac{A(n)}{d'} \qquad\text{with } d\mid d'\mid d^2.
\end{equation}

On the other hand, according to Fact~(i), one can write
$$
T_n^*\mathcal{K}_F=(q)(u,\sqrt{f(n)}-v)
$$
where $u\mid (v^2-f(n))$. By Fact~(i) again, the inclusion \eqref{eq:inclusion} implies that $q\mid d$ and that $qu\mid da'$, hence
$$
|u| \leq |da'| \leq |A(n)|,
$$
where the second inequality follows from \eqref{doud2}.

Finally, the quotient $da'/qu$ is a number whose prime factors divide $d(n)$, because the two ideals in \eqref{eq:inclusion} agree after inverting $d(n)$. Observing that $q\mid d\mid d(n)$, it follows that $a'$ and $u$ agree up to prime factors dividing $d(n)$, hence
$$
\left|\frac{a'}{\prod_{p\mid d(n)} p^{v_p(a')}}\right| \leq |u|.
$$
But $d\mid d(n)$, hence \eqref{doud2} yields
$$
\left|\frac{A(n)}{\prod_{p\mid d(n)} p^{v_p(A(n))}}\right|=\left|\frac{a'}{\prod_{p\mid d(n)} p^{v_p(a')}}\right|,
$$
and the result follows.
\end{proof}


\subsection{Proof of Theorems~\ref{thm1} and \ref{thm2}}

Let $h\in\Q[x]$ be an integer-valued polynomial. Then its \emph{fixed divisor} is by definition
$$
\fd(h):=\gcd\{h(n) ~|~ n\in\Z\}.
$$

If $h$ belongs to $\Z[x]$ then $\cont(h)$ clearly divides $\fd(h)$, but there are cases when the two are not equal: for example, the polynomial
$$
h(x)=x(x+1)(x+2)
$$
has content one (it's monic), and fixed divisor $6$ (a product of three consecutive integers is divisible by $6$). In general, it was proved by P{\'o}lya \cite{polya15} that, if $h$ is a polynomial with $\cont(h)=1$, then $\fd(h)$ divides $\deg(h)!$.

\begin{lem}
\label{fd_quad}
Let $\mathcal{L}\in\Pic(\W)$, with generic fiber $L\neq 0$. Let
$$
\left[\frac{A}{e},2\frac{B}{e},\frac{C}{e}\right]
$$
be the quadratic form as in Lemma~\ref{altmumford} whose class represents $L$. Let $d_L=\gcd(\fd(A),e)$ and let $\Delta_L$ be the product of primes dividing $e$. Then there exists an integer $N_L$ such that, for all $n\leq n_f$ satisfying $n\equiv N_L \pmod{\Delta_L}$, $T_n^*\mathcal{L}$ is the class of a quadratic form of the type
$$
[u(n),2v(n),w(n)],
$$
where $u(n)$ satisfies
$$
\left|\frac{A(n)}{d_L}\right| \leq |u(n)| \leq |A(n)|.
$$
\end{lem}

\begin{proof}
By construction, the polynomial
$$
\frac{A(x)}{d_L}
$$
has fixed divisor coprime to $e$. Therefore, for each prime $p$ dividing $e$ there exists an integer $m_p$ such that $A(m_p)/d_L\not\equiv 0\pmod{p}$. It follows from the Chinese remainder theorem that there exists an integer $N_L$ such that
$$
n\equiv N_L \pmod{\Delta_L} \quad\Longrightarrow\quad \gcd\left(\frac{A(n)}{d_L},e\right)=1.
$$
For such $n$, $d_L=\gcd(A(n),e)$. The result then follows from Lemma~\ref{quadlemma}.
\end{proof}

\begin{lem}
\label{keylemma}
Let us consider a family of quadratic forms $F_n=[a(n),2b(n),c(n)]$ with discriminant $4f(n)$, indexed by some subset $\Lambda\subseteq\Z$. Assume that there exists a non-constant polynomial $h\in\Q[x]$ with $\deg(h)<\deg(f)$, and an integer $M\geq 1$ such that
\begin{equation}
\label{hypH}
\forall n\in\Lambda,\quad \left|\frac{h(n)}{M}\right| \leq |a(n)| \leq |h(n)|.
\end{equation}
Then, for $n\in\Lambda$ negative enough, $F_n$ is not equivalent to the identity form $[1,0,-f(n)]$.
\end{lem}

The proof is similar to Soleng's proof \cite[Theorem~4.1]{soleng94}. The main idea is that, if two quadratic forms over $\Z$ have small coefficients in $X^2$ compared to their discriminant, then they are not equivalent.

\begin{proof}
Let $n\leq n_f$ be an integer in $\Lambda$ such that $F_n$ is equivalent to the identity. Then there exists
$\sigma=\left(\begin{smallmatrix}
    \alpha & \beta\\
    \gamma & \delta
\end{smallmatrix} \right)\in \GL_2(\Z)$
such that
$$
X^2-f(n)Y^2=\pm F_n(\alpha X+\beta Y,\gamma X+\delta Y).
$$
This implies that
\begin{equation}
\label{eqn1}
\pm 1=a(n)\alpha^2+2b(n)\alpha\gamma +c(n)\gamma^2.
\end{equation}
If $\gamma=0$, then $\alpha=\pm 1$ (because $\det(\sigma)=\pm 1$), hence the previous equality becomes
$$
a(n)=\pm 1.
$$
According to the double inequality \eqref{hypH}, this implies that
$$
1\leq |h(n)| \leq M.
$$

The polynomial $h$ being non-constant, such an equality may hold only for finitely many values of $n$. Hence, for $n\leq n_0$ where $n_0$ is an integer which depends only on $h$, we can deduce that $\gamma\neq 0$.

Multiplying by $a(n)$ on both sides of \eqref{eqn1}, we get
$$
\pm a(n)=a(n)^2\alpha^2+2a(n)b(n)\alpha\gamma +a(n)c(n)\gamma^2.
$$
Since $b^2-ac=f$, we deduce that
\begin{align*}
\pm a(n) &=a(n)^2\alpha^2+2a(n)b(n)\alpha\gamma +b(n)^2\gamma^2-f(n)\gamma^2 \\
 &= (a(n)\alpha+b(n)\gamma)^2-f(n)\gamma^2.
\end{align*}
Since a square is positive, it follows that
$$
\pm a(n) + f(n)\gamma^2 \geq 0.
$$
Since $\gamma\neq 0$, we have that $\gamma^2\geq 1$. But by assumption $n\leq n_f$, hence $f(n)$ is strictly negative (by the very definition of $n_f$), from which we deduce that
$$
\pm a(n) + f(n) \geq \pm a(n) + f(n)\gamma^2 \geq 0.
$$
It then follows from \eqref{hypH} that
$$
\pm h(n) + f(n) \geq 0.
$$

Finally, $\deg(h)< \deg(f)=2g+1$, hence $\pm h+f$ is a monic polynomial of degree $2g+1$, and there exists $n_1$ such that, for all $n\leq n_1$, $\pm h(n) + f(n)<0$. Therefore, if the integer $n$ is smaller than $\min(n_f,n_0,n_1)$, then $F_n$ cannot be equivalent to the identity.
\end{proof}

\begin{proof}[Proof of Theorem~\ref{thm1}]
Let $\mathcal{L}\in\Pic(\W)$ which is non-trivial, and let $L\in\Pic(C\setminus\{0\})$ be its generic fiber, which is also non-trivial according to Lemma~\ref{restriction}. Let
$$
\left[\frac{A}{e},2\frac{B}{e},\frac{C}{e}\right]
$$
be the quadratic form as in Lemma~\ref{altmumford} whose class represents $L$. Then according to Lemma~\ref{fd_quad}, there exists a congruence class $\Lambda_L\subseteq \Z$ and an integer $d_L$ such that, for all $n\in \Lambda_L$, $T_n^*\mathcal{L}$ is the class of a quadratic form $[u(n),v(n),w(n)]$ where $u(n)$ satisfies
$$
\left|\frac{A(n)}{d_L}\right| \leq |u(n)| \leq |A(n)|.
$$

We note that $A$ is a non-constant polynomial, because $L\neq 0$. According to Lemma~\ref{altmumford}, we have that $\deg(A)\leq g < 2g+1=\deg(f)$. Therefore, if we put $h(x):=A(x)$, then the family of quadratic forms above satisfies the hypotheses of Lemma~\ref{keylemma}, hence for $n$ negative enough in the congruence class $\Lambda_L$ we have that $T_n^*\mathcal{L}$ is not zero. There are infinitely many such integers, hence the result. In particular, if $\mathcal{L}$ has infinite order, then the order of $T_n^*\mathcal{L}$ is unbounded when $n\to -\infty$. We note that, when $k$ is enlarged, the denominator $e$ of the quadratic form representing the line bundle $\mathcal{L}^k$ is also enlarged, so there are fewer integers satisfying the congruence condition, and one has to choose a larger $n$ in order to ensure that $T_n^*\mathcal{L}^k$ is non-trivial.
\end{proof}

\begin{proof}[Proof of Theorem~\ref{thm2}]
Let $Q\in\J^0(\Z)$ be a point of infinite order, and let $L_Q\in\Pic^0(C)$ be the corresponding line bundle on $C$. Let $N_Q$ be the restriction of $L_Q$ to $C\setminus\{\infty\}$, then according to Lemma~\ref{restriction}, up to replacing $Q$ by some multiple, we may assume that $N_Q$ extends into a line bundle $\mathcal{N}_Q\in\Pic(\W)$. By definition $P_n$ is the generic fiber of $T_n$, hence the corresponding integral section $P_n:\Spec(\mathcal{O}_{\Q\left(\sqrt{f(n)}\right)})\to\W$ is obtained by normalisation of $T_n$. By Corollary~\ref{C:pairing}, the class group pairing $\langle P_n,Q\rangle^{\cl}$ is equal to $P_n^*\mathcal{N}_Q$. According to the previous discussion, $\langle P_n,Q\rangle^{\cl}$ is also equal to the image of $T_n^*\mathcal{N}_Q$ by the ``normalisation'' map
\begin{equation}
\label{bounded_ker}
\Pic(\Z[\sqrt{f(n)}])\rightarrow \Pic(\mathcal{O}_{\Q\left(\sqrt{f(n)}\right)}).
\end{equation}

It is a classical fact \cite[Chap.~2, \S{}7, D]{cox89} that the kernel of this map is a quotient of
$$
\left(\mathcal{O}_{\Q\left(\sqrt{f(n)}\right)}/m\right)^{\times}
$$
where $m\in\Z$ denotes the conductor of $\Z[\sqrt{f(n)}]$ as an order in the ring of integers $\mathcal{O}_{\Q\left(\sqrt{f(n)}\right)}$. Let $S(n)$ be the largest integer whose square divides $f(n)$. Then $m=S(n)$ or $m=2S(n)$ depending on the congruence class of $f(n)/S(n)^2$ modulo $4$. In any case, the size of the kernel of the map \eqref{bounded_ker} is bounded above by $4S(n)^2$.

For simplicity, we set $L:=N_Q$, and $\mathcal{L}:=\mathcal{N}_Q$. Let
$$
\left[\frac{A}{e},2\frac{B}{e},\frac{C}{e}\right]
$$
be the quadratic form as in Lemma~\ref{altmumford} whose class represents $L$. According to the discussion above, in order to prove the theorem it suffices to prove that there exists some $n$ for which $T_n^*\mathcal{L}$ has order strictly larger than the size of the kernel of \eqref{bounded_ker}.

(i) According to a classical result of Hooley \cite{hooley68}, if $f$ is a non-constant irreducible polynomial of degree at most three, then there exist a positive density of $n\in\Z$ for which $f(n)/\fd(f)$ is square-free. According to Booker and Browning \cite{BB16}, this result extends to the case when $f$ is a square-free product of such polynomials.

More precisely, let
$$
f=\prod_{i\in I} f_i
$$
be the factorization of $f$ in product of irreducible factors in $\Q[x]$. Because $f$ belong to $\Z[x]$ and is monic, the same holds for the $f_i$. A first step in the proof of \cite{BB16} is to eliminate bad primes for $f$, that is, primes $\leq \deg(f)$ and primes dividing the resultants $\Res(f_i,f_j)$ for $i\neq j$. Let $\Delta_f$ denote the product of bad primes for $f$, then there exists an integer $\alpha_f$ such that
$$
n\equiv \alpha_f \pmod{\Delta_f} \quad\Longrightarrow\quad \gcd\left(\frac{f(n)}{\fd(f)},\Delta_f\right)=1.
$$

We note that, because $\cont(f)=1$, the prime factors of $\fd(f)$ are $\leq \deg(f)$, hence divide $\Delta_f$. So whenever $n\equiv \alpha_f \pmod{\Delta_f}$, $\frac{f(n)}{\fd(f)}$ is coprime to $\fd(f)$.

It is shown in \cite{BB16} that there exist a positive density of $n\equiv \alpha_f \pmod{\Delta_f}$ such that $\frac{f(n)}{\fd(f)}$ is square-free. For such $n$, the largest integer $S(n)$ whose square divides $f(n)$ is equal to the largest integer whose square divides $\fd(f)$, hence $S(n)^2\leq \fd(f)$.

We shall now combine this result with a suitable analogue of Lemma~\ref{fd_quad}. Let $\Delta_{L,f}$ be the product of primes dividing $e$, primes $\leq \deg(A)+\deg(f)$ and primes dividing the resultants $\Res(f_i,f_j)$ for $i\neq j$. Then there exists $\alpha_{L,f}$ such that
$$
n\equiv \alpha_{L,f} \pmod{\Delta_{L,f}} \quad\Longrightarrow\quad \gcd\left(\frac{A(n)f(n)}{\fd(Af)},\Delta_f\right)=1.
$$

One can show along the lines of \cite{BB16} that there exist a positive density of $n\equiv \alpha_{L,f} \pmod{\Delta_{L,f}}$ such that $\frac{f(n)}{\fd(f)}$ is square-free. As pointed previously, for such $n$ we have that $S(n)^2\leq \fd(f)$, hence the kernel of the map \eqref{bounded_ker} is bounded above by $4\fd(f)$, according to the discussion at the beginning of the proof.

On the other hand, one can generalize Lemma~\ref{fd_quad} using this new congruence class $\Lambda_{L,f}$, and replacing $d_L=\gcd(\fd(A),e)$ by $d_{L,f}:=\gcd(\fd(Af),e)$.
Then, according to Lemma~\ref{keylemma}, for $n\in\Lambda_{L,f}$ negative enough, the class $T_n^*\mathcal{L}$ is not zero. Replacing $L$ by some multiple, we find that the order of $T_n^*\mathcal{L}$ is strictly larger than $4\fd(f)$ when $n\in\Lambda_{L,f}$ is negative enough. It follows from the discussion above that, for a positive density of such $n$, its image by the map \eqref{bounded_ker}, which is none other than $P_n^*\mathcal{L}$, is not trivial. Replacing again $L$ by its multiples shows that the order of $P_n^*\mathcal{L}$ is unbounded, hence the result.

(ii) According to Granville \cite{granville98}, the abc conjecture implies that, for any square-free polynomial $f$, there exists a positive density of $n$ such that $f(n)/\fd(f)$ is square-free. This allows to extend the result to such polynomials.
\end{proof}

\begin{proof}[Proof of Corollary~\ref{cor1}]
Let $\varphi:J_1\times J_2\to V$ be an isogeny, where $J_1$ and $J_2$ are Jacobians of hyperelliptic curves satisfying the hypotheses of Theorem~\ref{thm2} (we consider for simplicity the case when the product contains two factors, the general case follows). Let $\varphi^t:V^t\to J_1\times J_2$ be the dual isogeny, which induces a map $\varphi^t:\V^t\to \J_1\times \J_2$ between the corresponding N{\'e}ron models. If $Q$ belongs to $\V^{t,0}$, then its image by $\varphi^t$ belongs to $\J_1^0\times \J_2^0$. Up to switching factors, we may assume that $\varphi^t(Q)=(Q_1,Q_2)$ where $Q_1\in\J_1^0(\Z)$ has infinite order. Then, according to Theorem~\ref{thm2}, there exists some point $P_n\in J_1(\overline{\Q})$ defined over an imaginary quadratic field, such that
$$
\langle P_n,Q_1\rangle^{\cl}\neq 0.
$$

On the other hand, we have
$$
\langle \varphi(P_n,0),Q \rangle^{\cl} = \langle (P_n,0),\varphi^t(Q) \rangle^{\cl} = \langle (P_n,0),(Q_1,Q_2)\rangle^{\cl} = \langle P_n,Q_1\rangle^{\cl}
$$
where the first equality follows from universal properties of Poincar{\'e} line bundles. The result follows by putting $P:=\varphi(P_n,0)$, which is defined over the same field as $P_n$.
\end{proof}


\section{Relation with the constructions of Buell and Soleng}
\label{sec:BS}


The so-called class group morphism $\delta$ on elliptic curves was first constructed by Buell \cite{buell76,buell77}, and rediscovered independently by Soleng \cite{soleng94}. Later, Soleng generalised his construction to hyperelliptic curves \cite{soleng11}.

It was conjectured by Mazur and Tate \cite{mt}, and proved by Call \cite{call86}, that this class group morphism is equal to Mazur-Tate's class group pairing  with one argument fixed. More recently, Buell and Call \cite{bc16} gave a new proof of this relation. 

The aim of this section is to extend this comparison result to hyperelliptic curves and their Jacobians. We feel that our argument, based on techniques developed in this paper, is more conceptual and gives a direct explanation why the two objects are related. In particular, we replace the use of N{\'e}ron local symbols, which was central in the previous proofs, by the description of the Picard group of $\W$ in terms of quadratic forms. Our proof is global by nature, and relies on the definition of the class group pairing in terms of biextensions (or Poincar{\'e} line bundles).

Let us recall briefly how Soleng defined the class group morphism in \cite{soleng11} for an hyperelliptic curve $C$. We consider the same notation as in Theorem~\ref{thm1}; in particular, we choose an odd degree Weierstrass equation for $C$ of the form
$$
y^2=f(x),
$$
which gives rise to an affine integral model $\W\to\Spec(\Z)$. Given $n\in\Z$ such that $f(n)$ is not a square, we shall consider the quadratic section $T_n=(n,\sqrt{f(n)})$ of $\W$, and its generic fiber $P_n$.

The following definition is extracted from \cite{soleng11}. It contains the condition which is needed in order to be able to define $\delta_n$.

\begin{dfn}
\label{dfn:primitive}
Let $n\in\Z$ and $Q\in J(\Q)$. We say that $Q$ is \emph{$n$-primitive} if there exists a quadratic form $F=\left[\frac{A}{e},2\frac{B}{e},\frac{C}{e}\right]$ whose class represents $N_Q$, where $A$, $B$ and $C$ belong to $\Z[x]$ and $e>0$ is an integer, satisfying the condition
\begin{equation}
\label{eq:primitive}
\gcd(A(n),2B(n),C(n))=1.
\end{equation}
Then we let
$$
\delta_n(Q):= (A(n),e\sqrt{f(n)}-B(n))\Z[\sqrt{f(n)}].
$$
\end{dfn}

\begin{rmq}
\label{rmq:various}
\begin{enumerate}
\item[a)] The condition \eqref{eq:primitive} just means that the quadratic form $[A(n),2B(n),C(n)]$, with discriminant $4e^2f(n)$, is primitive, hence its class defines an element in $\Pic(\Z[e\sqrt{f(n)}])$. By definition, $\delta_n(Q)$ is the image of this ideal by the natural map $\Pic(\Z[e\sqrt{f(n)}])\rightarrow \Pic(\Z[\sqrt{f(n)}])$.
\item[b)] One of the most technical points in \cite{soleng11} is to prove that $\delta_n(Q)$ only depends on $Q$, not on the choice of $F$. We give a new proof of this fact below.
\item[c)] The concept of $n$-primitivity is not canonical in the sense that it depends on the choice of the equation $y^2=f(x)$ for the curve $C$.
\item[d)] If $Q$ is $n$-primitive, then with the notation above we have
$$
\gcd(A(n),e,C(n))=1.
$$
Indeed, according to the relation $B^2-AC=e^2f$, if some prime number divides $A(n)$, $e$ and $C(n)$, then it also divides $B(n)$.
\item[e)] If $f(n)$ is square-free, then any $Q$ is $n$-primitive, for the same reason as in Remark~\ref{rmq:squarefree}.
\item[f)] In the first definition by Buell \cite{buell77}, only $\delta_0$ is considered, and $f(0)$ is assumed to be square-free. The definition of primitive class as above was given by Soleng in \cite{soleng94} for elliptic curves, and in \cite{soleng11} for hyperelliptic curves.
\end{enumerate}
\end{rmq}

Using the tools developed in this paper, we shall now translate in our language, and give a new proof, of the main results of Buell \cite{buell77} and Soleng \cite{soleng11}, namely: the set of $n$-primitive points is a subgroup of $J(\Q)$, and the map $\delta_n$ is a morphism.

\begin{thm}
\label{thm:buellsoleng}
Let us fix an integer $n$, and let $\Sigma_n$ be the set of singular points of $\W$ which do not lie on the subscheme $T_n$. Then
\begin{enumerate}
\item[1)] A point $Q$ is $n$-primitive if and only if $N_Q$ belongs to the image of the (injective) map
$$
\Pic(\W\setminus\Sigma_n)\longrightarrow \Pic(C\setminus\{\infty\})\simeq J(\Q).
$$
In particular, the set of $n$-primitive points is a subgroup of $J(\Q)$, which contains $\J^0(\Z)$;
\item[2)] If $Q$ is $n$-primitive, and if $\mathcal{N}_Q$ is the unique line bundle on $\W\setminus\Sigma_n$ extending $N_Q$, then
\begin{equation}
\label{eq:delta_n}
\delta_n(Q)=T_n^*\mathcal{N}_Q.
\end{equation}
In particular, the map $\delta_n$ is well-defined, and is a morphism from the subgroup of $n$-primitive classes to $\Pic(\Z[\sqrt{f(n)}])$.
\end{enumerate}
\end{thm}

Thanks to the following statement \cite[Lemma~3.1]{soleng11}, that we recall for the reader's convenience, one may assume that $\gcd(A(n),e)=1$ when defining $\delta_n$.

\begin{lem}
\label{lem:soleng3.1}
Let $Q$ be a $n$-primitive point, and let $F=\left[\frac{A}{e},2\frac{B}{e},\frac{C}{e}\right]$ be a quadratic form as in Definition~\ref{dfn:primitive} whose class represents $N_Q$. Then there exists a quadratic form $\left[\frac{A'}{e},2\frac{B'}{e},\frac{C}{e}\right]$ as in Definition~\ref{dfn:primitive}, which is equivalent to $F$, such that $\gcd(A'(n),e)=1$. Moreover,
$$
(A(n),e\sqrt{f(n)}-B(n))\Z[\sqrt{f(n)}]=(A'(n),e\sqrt{f(n)}-B'(n))\Z[\sqrt{f(n)}].
$$
\end{lem}

\begin{proof}
According to Remark~\ref{rmq:various}, $\gcd(A(n),e,C(n))=1$. Let us write $e=e_1e_2$ where $\gcd(e_1,e_2)=\gcd(A(n),e_1)=\gcd(C(n),e_2)=1$. Up to enlarging $e_1$, we may assume that all prime factors of $e_2$ divide $A(n)$. Since $B(n)^2-A(n)C(n)=e^2f(n)$, all prime factors of $e_2$ divide also $B(n)$.

We have
$$
\left[\frac{A}{e},2\frac{B}{e},\frac{C}{e}\right]\sim \left[\frac{A+2e_1B+e_1^2C}{e},2\frac{B+e_1C}{e},\frac{C}{e}\right],
$$
and we claim that $\gcd(A(n)+2e_1B(n)+e_1^2C(n),e)=1$. Indeed,
$$
\gcd(A(n)+2e_1B(n)+e_1^2C(n),e_1)=\gcd(A(n),e_1)=1.
$$
On the other hand, if a prime $p$ divides $e_2$, then it also divides $A(n)$ and $B(n)$ by construction. If in addition $p$ divides $A(n)+2e_1B(n)+e_1^2C(n)$, then $p$ divides $e_1C(n)$. But $p$ cannot divide $e_1$ nor $C(n)$, because $e_2$ is coprime to these numbers. It follows that
$$
\gcd(A(n)+2e_1B(n)+e_1^2C(n),e_2)=1,
$$
hence the result.
\end{proof}

\begin{proof}[Proof of Theorem~\ref{thm:buellsoleng}]
According to Lemma~\ref{lem:soleng3.1}, if $Q$ is $n$-primitive, then there exists a quadratic form $F=\left[\frac{A}{e},2\frac{B}{e},\frac{C}{e}\right]$ as in Definition~\ref{dfn:primitive} whose class represents $N_Q$, and such that $\gcd(A(n),e)=1$. As Soleng already pointed out \cite[Proof of 3.4]{soleng11}, this implies that $\cont(A)=1$. Indeed, if some prime $p$ divides $\cont(A)$, then $B^2\equiv e^2f \pmod{p}$. But $e$ is coprime to $A(n)$, hence to $\cont(A)$, and in particular $e$ is invertible mod $p$. It follows that the reduction of $f$ modulo $p$ is the square of a polynomial in $\mathbb{F}_p[x]$, which is impossible since its degree is odd.

Let $\mathcal{Z}_F\subset \W$ be the closed subscheme defined by the ideal
$$
\mathcal{I}_F:=(A,ey-B).
$$
The generic fiber of $\mathcal{Z}_F$ is the divisor $D_F$ corresponding to the quadratic form $F$. The same argument as in the proof of Lemma~\ref{lem:closure} shows that
\begin{equation}
\overline{D_F}\subseteq \mathcal{Z}_F
\end{equation}
and these two subschemes agree on the open subset of $\W$ over which $e$ is invertible.

On the other hand, because $A(n)$ is coprime to $e$, the closed subschemes $T_n$ and $\mathcal{Z}_F$ have empty intersection above places dividing $e$. Thus, $T_n$ and $\overline{D_F}$ have empty intersection above primes dividing $e$. At a prime $p$ not dividing $e$, the condition \eqref{eq:primitive} ensures us that $\mathcal{Z}_F$, hence $\overline{D_F}$, is a divisor in the neighbourhood of each point in $T_n\cap \W_{\mathbb{F}_p}$. Let us check this more carefully: if $p=2$ then, because $\gcd(A(n),2B(n),C(n))=1$, we may assume (up to exchanging $A$ and $C$) that $A(n)$ is odd, in which case $T_n$ does not intersect $\overline{D_F}$. Let us now assume that $p\neq 2$, then we consider two cases: if $p$ does not divide $f(n)$, then $T_n$ does not meet the singular locus of $\W_{\mathbb{F}_p}$, hence $\overline{D_F}$ is a divisor in the neighbourhood of $T_n\cap \W_{\mathbb{F}_p}$. If $p$ divides $f(n)$, then $p$ cannot divide both $A(n)$ and $C(n)$ because it would then divide $B(n)$, which is impossible by \eqref{eq:primitive}; therefore, up to exchanging $A$ and $C$, we may assume that $p$ does not divide $A(n)$, which implies that $\overline{D_F}$ does not meet $T_n$ above $p$, hence the result.

Now, going through the proof of Lemma~\ref{restriction}, one see that $N_Q$ extends into a line bundle on $\W\setminus\Sigma_n$ if and only if $\overline{D_F}$  is a divisor on $\W\setminus\Sigma_n$. Therefore, according to the reasoning above, if $Q$ is $n$-primitive then $N_Q$ (uniquely) extends into a line bundle $\mathcal{N}_Q$ on $\W\setminus\Sigma_n$, and we have
$$
\delta_n(Q)=T_n^*\mathcal{N}_Q,
$$
which proves that $\delta_n$ is well-defined, and is a group morphism.

It remains to prove that, if $N_Q$ extends into a line bundle $\mathcal{N}_Q$ on $\W\setminus\Sigma_n$---or equivalently, if $\overline{D_F}$ is a divisor on $\W\setminus\Sigma_n$---then $Q$ is $n$-primitive. Let us pick a quadratic form $F=[A,2B,C]$ with discriminant $4f$ representing $N_Q$, where $A$, $B$ and $C$ belong to $\Z[x]$. Then one checks that $\mathcal{Z}_F=\overline{D_F}+V$ where $V$ is some vertical divisor (if $\cont(A)=1$ then $V=\emptyset$). If $\overline{D_F}$ is a divisor on $\W\setminus\Sigma_n$, then $\mathcal{Z}_F$ is also a divisor, and because vertical divisors are principal, $\mathcal{N}_Q$ is the class of $\mathcal{Z}_F$ in the Picard group of $\W\setminus\Sigma_n$. So $T_n^*\mathcal{N}_Q$ is the ideal class of 
$$
T_n^*\mathcal{I}_F=(A(n),\sqrt{f(n)}-B(n)).
$$

This proves that the quadratic form $[A(n),2B(n),C(n)]$, with discriminant $4f(n)$, is primitive because its corresponding ideal is invertible.
\end{proof}

The statement below is a generalisation of the result of Buell and Call \cite[Theorem~3.14]{bc16} to hyperelliptic curves. Under good reduction properties relative to the model $\W$, it allows one to compute values of the class group pairing in terms of quadratic forms.

\begin{cor}
Let $n$ be an integer, and let $Q\in J(\Q)$ such that, for each prime $p$, either $T_n$ reduces in the smooth locus of $\W_{\mathbb{F}_p}$, or $N_Q$ extends into a line bundle in the neighbourhood of each singular point of $\W_{\mathbb{F}_p}$. Then:
\begin{enumerate}
\item[1)] $Q$ is $n$-primitive
\item[2)] $(P_n,Q)$  is a good reduction pair
\item[3)] $\langle P_n,Q \rangle^{\cl}$ is the image of $\delta_n(Q)$ by the natural map $\Pic(\Z[\sqrt{f(n)}])\rightarrow \Pic(\mathcal{O}_{\Q\left(\sqrt{f(n)}\right)})$.
\end{enumerate}
\end{cor}

Let us note that, if $f(n)$ is square-free, then $T_n$ belongs to the smooth locus of $\W\to\Spec(\Z)$, hence the condition above is automatically satisfied. We note that previous comparison results \cite{call86} and \cite{bc16} were obtained under the assumption that $f(n)$ is square-free. Thus, even in the case when $C$ is an elliptic curve, the statement above improves on previous results.

\begin{proof}
1) Obvious by the very definition of $n$-primitive point.

2) Let $\overline{\W}^\mathrm{sm}$ be the smooth locus of $\overline{\W}$. By the universal property of the N{\'e}ron model, the embedding $\iota:C\to J$ extends uniquely into a map $\overline{\W}^\mathrm{sm}\to\J$. This map factors through $\J^0$, because the fibers of $\overline{\W}^\mathrm{sm}$ are connected. On the other hand, if $N_Q$ extends into a line bundle on $\W$ then $Q$ belongs to $\J^0(\Z)$ according to Corollary~\ref{C:pairing}. Doing this reasoning fiber by fiber proves that $(P_n,Q)$ is a good reduction pair.

3) Here we work over $\Spec(\mathcal{O}_K)$, but for simplicity we still denote by $\W$, $\J$, etc. the integral models over $\Spec(\mathcal{O}_K)$. Let $\omega^\mathrm{sm}:\overline{\W}^\mathrm{sm}\to\J^0$ be the unique map extending $\iota:C\to J$.
Using the same notation as in \S{}\ref{subsec:MT}, we denote by $\Poinc^0_\mathrm{left}$ the Poincar{\'e} line bundle on $\J^0\times \J$. Then, for each $Q\in J(\Q)$, the unique line bundle $\mathcal{M}_Q$ on $\overline{\W}^\mathrm{sm}$ extending $L_Q$ satisfies
$$
\mathcal{M}_Q=(\omega^\mathrm{sm}\times Q)^*\Poinc^0_\mathrm{left}
$$
because these two line bundles are rigidified, and coincide on the generic fiber.

On the other hand, if $N_Q$ extends into a line bundle on $\W$, then $L_Q$ extends into a line bundle $\mathcal{M}_Q$ on $\overline{\W}$, hence according to Lemma~\ref{L:ConCom} the point $Q$ belongs to $\J^0(\Z)$. Let $h:\X\to \overline{\W}$ be the minimal desingularisation of $\overline{\W}$. It follows from  Lemma~\ref{L:ConCom} and Lemma~\ref{lem:classOK} that
$$
h^*\mathcal{M}_Q|_{\X^\mathrm{sm}}=(\iota^\mathrm{sm}\times Q)^*\Poinc^0_\mathrm{right}
$$
where $\iota^\mathrm{sm}:\X^\mathrm{sm}\to\J$ is the unique map extending $\iota:C\to J$.

Let $S_1$ be the set of places at which $N_Q$ extends into a line bundle in the neighbourhood of each singular point of the fiber of $\W$ at $v$, and let $S_2$ be the complement of $S_1$ in the set of places of bad reduction of $\J$.

Let $\X^1$ be the open subscheme of $\X$ whose special fiber at $v$ is $\X$ for $v\in S_1$, and $\overline{\W}$ otherwise. Let $h^1:\X^1\rightarrow\overline{\W}$ be the restriction of $h$, and let $\iota^{1,\mathrm{sm}}:\X^{1,\mathrm{sm}}\to \J^1$ be the map extending $\iota$. Then it follows from the definition of $\Poinc^{1,2}$ and the reasoning above that
$$
(h^1)^*\mathcal{M}_Q|_{\X^{1,\mathrm{sm}}}=(\iota^{1,\mathrm{sm}}\times Q)^*\Poinc^{1,2}.
$$

On the other hand, it follows from our assumptions that $P_n$ extends into a section of $\X^{1,\mathrm{sm}}$. Therefore, we can compute $P_n^*\mathcal{M}_Q$ after pulling-back $\mathcal{M}_Q$ by the map $h^1$, which yields, according to the relation above,
$$
P_n^*\mathcal{M}_Q=(\iota(P_n)\times Q)^*\Poinc^{1,2}=\langle P_n,Q \rangle^{\cl}.
$$

Finally, $P_n$ being the generic fiber of $T_n$ it gives rise to an integral section of $\W$, hence
$$
P_n^*\mathcal{M}_Q=P_n^*\mathcal{N}_Q,
$$
and the result follows by \eqref{eq:delta_n}.
\end{proof}



\bibliographystyle{amsalpha}
\bibliography{biblio.bib}



\bigskip

\textsc{Jean Gillibert}, Institut de Math{\'e}matiques de Toulouse, CNRS UMR 5219, 118 route de Narbonne, 31062 Toulouse Cedex 9, France.

\emph{E-mail address:} \texttt{jean.gillibert@math.univ-toulouse.fr}


\end{document}